\documentclass[12pt, a4paper]{amsart}
\usepackage{hyperref, fullpage, amsmath, amsfonts, amsthm, amssymb, stmaryrd}

\theoremstyle{plain}
\newtheorem{Theorem}{Theorem}[section]

\newtheorem{Corollary}[Theorem]{Corollary}
\newtheorem{Proposition}[Theorem]{Proposition}

\theoremstyle{remark}
\newtheorem{Remark}[Theorem]{Remark}

\newcommand{\hhb}[1]{{\hbox to#1pt{}}}

\newcommand{\mfp}{\mathfrak{p}}

\address{Institut f\"ur Algebra, Technische Universit\"at Dresden, 01062 Dresden}
\email{arno.fehm@tu-dresden.de}

\begin{document}

\title{Three counterexamples concerning the \\Northcott property of fields}
\author{Arno Fehm}
\maketitle

\begin{abstract}
We give three examples of fields concerning the Northcott property on elements of small height:
The first one has the Northcott property but its Galois closure does not even satisfy the Bogomolov property.
The second one has the Northcott property and is pseudo-algebraically closed, i.e.~every variety has a dense set of rational points.
The third one has bounded local degree at infinitely many rational primes but does not have the Northcott property.
\end{abstract}

\section{Introduction}

\noindent
Northcott's theorem on the finiteness of elements of bounded height in number fields is of central importance in diophantine geometry,
for example very classically in the proof of the Mordell-Weil theorem.
Motivated by that, Bombieri and Zannier \cite{BZ}
say that a field $K\subseteq\overline{\mathbb{Q}}$ has the {\bf Northcott property (N)} if for each $T>0$
the set 
$$
 K_T:=\{\alpha\in K^\times:h(\alpha)<T\}
$$ 
is finite, where $h:\overline{\mathbb{Q}}\rightarrow\mathbb{R}$ denotes the absolute logarithmic Weil height.
In the same paper, the authors introduce another closely related notion:
A field $K$ has the {\bf Bogomolov property (B)} if there exists $T>0$ such that
$K_T$ consists only of the roots of unity in $K$.
Note that clearly (N) implies (B).
These and related properties have since been studied by various authors, see e.g.~\cite{AZ, DZ, Widmer, CW, Habegger, Pottmeyer_p, GR}.

One theme in this area is whether properties like (N) and (B) are preserved under taking Galois closures.
For example, \cite[Cor.~2]{Widmer} gives a field $K\subseteq\overline{\mathbb{Q}}$ with (N) whose Galois closure over $\mathbb{Q}$ does not have (N).
Similarly, \cite[Example 1]{Pottmeyer} gives a field $K\subseteq\overline{\mathbb{Q}}$ with (B) whose Galois closure over $\mathbb{Q}$ does not have (B),
and states that ``It would be interesting to know whether the Galois closure of
a field with the Northcott property necessarily satisfies the Bogomolov property.''
Our first result is that the answer to this is negative:

\begin{Proposition}\label{prop1}
There exists an algebraic extension $K/\mathbb{Q}$ 
such that $K$ has the Northcott property
but the Galois closure of $K/\mathbb{Q}$ does not have the Bogomolov property.
\end{Proposition}

The intuition being that varieties over fields with (B), or even more so, with (N), have `few' point,
Amoroso, David and Zannier \cite{ADZ} 
asked whether there exists a field $K$ with (B) that
is {\bf pseudo-algebraically closed}, i.e.~every geometrically irreducible variety $V$ over $K$ has a $K$-rational point\footnote{This property first occurred in the work of Ax on the elementary theory of finite fields. The term {\em pseudo-algebraically closed} was coined by Frey.},
and they present ``some evidences for a negative answer''.
However, Pottmeyer \cite{Pottmeyer} showed that such fields do exist,
and while this was seen as surprising,
it was apparently expected that at least there should be no pseudo-algebraically closed fields with (N):
Our second result is that such fields do in fact exist,
and can even be chosen Galois over $\mathbb{Q}$ (which might be interesting in light of Proposition~\ref{prop1}):

\begin{Proposition}\label{prop2}
There exists a Galois extension $K/\mathbb{Q}$
such that $K$ is pseudo-algebraically closed and has the Northcott property.
\end{Proposition}

As the Northcott property implies a variety of other well-studied properties of fields
(see e.g.~\cite[Theorem 6.8]{CW}),
for example on pre-periodic points of polynomial mappings,
this proposition might also give surprising counterexamples to some of the questions there,
but we will not discuss these implications here.

The construction of the example in Proposition \ref{prop1} is completely elementary,
while the construction of the example in Proposition \ref{prop2} 
uses some (known) results on specializations of covers of curves.
The Northcott property in both cases follows from a very general criterion of Widmer \cite{Widmer},
which we recall is Section \ref{sec:Widmer}.

Pottmeyer \cite[Question 4.7]{Pottmeyer_p} asks whether the Northcott property is implied
by other properties like the Narkiewicz property (R) (cf.~\cite[Definition 6.6]{CW}).
The following example answers this questions negatively:

\begin{Proposition}\label{prop3}
There exists a Galois extension $K/\mathbb{Q}$ such that infinitely many prime numbers 
are totally split in $K$ but $K$ does not have the Northcott property.
\end{Proposition}

Namely, Pottmeyer \cite[Theorem 4.3]{Pottmeyer_p} shows that every Galois extension of $\mathbb{Q}$ that has finite local degree
at infinitely many prime numbers (in particular, any $K$ as in Proposition \ref{prop3})
satisfies the so-called {\em universal strong Bogomolov property (USB)}, 
which in turn implies (R) and other related properties \cite[Lemma 4.2]{Pottmeyer_p}.
The construction of the example in Proposition \ref{prop3} builds on a result of Bombieri and Zannier \cite{BZ}.

\section{Widmer's criterion}\label{sec:Widmer}

\noindent
We start by quoting the criterion of Widmer \cite[Theorem 3]{Widmer}
and state a special case that is sufficient for our constructions:

\begin{Theorem}\label{thm:Widmer}
Let $K_0\subseteq K_1\subseteq\dots$ be a tower of number fields with
$$
 \inf_{K_{i-1}\subsetneqq M\subseteq K_i}N_{K_{i-1}/\mathbb{Q}}(D_{M/K_{i-1}})^{([M:K_0][M:K_{i-1}])^{-1}}\rightarrow\infty\mbox{ as }i\rightarrow\infty,
$$
where the infimum is taken over intermediate fields $M$, and $D_{M/K_i}$ denotes the relative discriminant.
Then $K:=\bigcup_{i=0}^\infty K_i$ has (N).
\end{Theorem}

\begin{Corollary}\label{cor}
Let $K_0\subseteq K_1\subseteq\dots$ be a tower of number fields
and let $d_i=[K_i:\mathbb{Q}]$.
If for each intermediate field $K_{i-1}\subsetneqq M\subseteq K_i$
there exists a prime number $p>i^{d_i^2}$
that is unramified in $K_{i-1}$ but ramified in $M$,
then $K:=\bigcup_{i=0}^\infty K_i$ has (N).
\end{Corollary}

\begin{proof}
Let $K_{i-1}\subsetneqq M\subseteq K_i$.
If $p$ is ramified in $M$ but not in $K_{i-1}$, there is a prime $\mathfrak{p}$ of $K_{i-1}$ over $p$ that ramifies in $M$.
Then $\mfp|D_{M/K_{i-1}}$, hence $p|N:=N_{K_{i-1}/\mathbb{Q}}(D_{M/K_{i-1}})$.
Thus $N\geq p>i^{d_i^2}$, hence
$N^{([M:K_0][M:K_{i-1}])^{-1}}\geq i$, so Theorem \ref{thm:Widmer} applies.
\end{proof}

\section{Proof of Proposition \ref{prop1}}

\noindent
Let $K_0$ be any proper finite extension of $\mathbb{Q}$ in $\overline{\mathbb{Q}}$.
We fix an algebraic integer $0\neq\alpha\in K_0$ and $\sigma\in{\rm Gal}(\overline{\mathbb{Q}}/\mathbb{Q})$ such that
$\beta:=\sigma\alpha/\alpha$ is not a root of unity
(this is always possible, but take for example $K_0=\mathbb{Q}(i)$, $\alpha=2+i$ and complex conjugation as $\sigma$).
Choose a sequence of prime numbers $l_i$ with $l_i\rightarrow\infty$
and let $d_i=[K_0:\mathbb{Q}]\cdot l_1\cdots l_i$.
We now construct a tower of number field $K_0\subseteq K_1\subseteq\dots$ with $[K_i:\mathbb{Q}]=d_i$.
Suppose we already constructed $K_0,\dots,K_{i-1}$.
Fix a prime number $p_i>i^{d_i^2}$ that in addition does not ramify in $K_{i-1}$
and does not divide $N_{K_0/\mathbb{Q}}(\alpha)$,
let $\gamma_i$ be an $l_i$-th root of $p_i\alpha$ in $\overline{\mathbb{Q}}$
and define $K_i=K_{i-1}(\gamma_i)$.

We claim that $K:=\bigcup_{i} K_i$ has the desired properties:
For each $i$, $p_i$ ramifies in $K_i$ but not in $K_{i-1}$, and there are no other intermediate fields $K_{i-1}\subsetneqq M\subseteq K_i$.
Therefore, Corollary \ref{cor} applies and gives that $K$ has (N).
However, if $\hat{K}$ denotes the Galois closure of $K$ over $\mathbb{Q}$, then for each $i$,
$\hat{K}$ contains both $\gamma_i$ and $\sigma\gamma_i$
and therefore also 
$\frac{\sigma\gamma_i}{\gamma_i}$,
which satisfies
$$
 \left(\frac{\sigma\gamma_i}{\gamma_i}\right)^{l_i}=\frac{\sigma(p_i\alpha)}{p_i\alpha}=\frac{\sigma\alpha}{\alpha}=\beta.
$$
So since $\beta$ is not a root of unity, 
neither is $\frac{\sigma\gamma_i}{\gamma_i}$,
and $h(\frac{\sigma\gamma_i}{\gamma_i})=\frac{1}{l_i}h(\beta)\rightarrow 0$, hence $\hat{K}$ does not satisfy (B).

\section{Proof of Proposition \ref{prop2}}

\noindent
We want to construct a certain field $K\subseteq\overline{\mathbb{Q}}$ and prove that it is pseudo-algebraically closed.
It is well-known that for this it suffices to show that every geometrically irreducible curve $X$ over $\mathbb{Q}$ has a $K$-rational point,
see \cite[Theorem 11.2.3]{FJ}.
Moreover, since every curve admits a finite cover which is itself a Galois cover of $\mathbb{P}^1$ (see \cite[Theorem 18.9.3]{FJ}),
it suffices to prove the statement for the latter curves.
Therefore, let $X_1,X_2,\dots$ be an enumeration of the geometrically irreducible curves over $\mathbb{Q}$
that admit a Galois morphism to $\mathbb{P}^1$.

For each $i$ we will construct a suitable finite Galois extension $N_i$ of $\mathbb{Q}$ of degree $n_i$ such that $X_i(N_i)\neq\emptyset$,
let $K_i$ be the compositum of $N_1,\dots,N_i$ 
(which has then degree at most $d_i:=n_1\cdots n_i$ over $\mathbb{Q}$)
and $K$ the union of the $K_i$ (i.e.~the compositum of all $N_i$).
Suppose we already constructed $N_1,\dots,N_{i-1}$ and denote by $K_{i-1}$ their compositum.
Fix a Galois morphism $\varphi_i:X_i\rightarrow\mathbb{P}^1$,
which induces a Galois extension of function fields $F:=\mathbb{Q}(\mathbb{P}^1)\subseteq \mathbb{Q}(X_i)=:E$.

We now apply a version of Hilbert's irreducibility theorem that allows some control on the ramification.
While there are several such results in the literature, we intend to use \cite[Corollary 3.3]{Legrand}.
For this,
list the 
intermediate fields\footnote{In fact, it would suffice to work with the {\em minimal} such fields.} $F\subsetneqq M\subseteq E$
that are Galois over $F$ as $M_1,\dots,M_{r}$
and observe that each $M_j/F$ ramifies in some branch point $\alpha_j\in\mathbb{A}^1(\overline{\mathbb{Q}})=\overline{\mathbb{Q}}$
by the Riemann-Hurwitz formula.
In particular, there is a corresponding inertia subgroup $I\subseteq{\rm Gal}(E/F)$ not contained in ${\rm Gal}(E/M_j)$.
Pick $g\in I\setminus{\rm Gal}(E/M_j)$ and let $C_j:=g^{{\rm Gal}(E/F)}$ be the conjugacy class of $g$.
If $m_j\in\mathbb{Q}[X]$ denotes the minimal polynomial of $\alpha_j$ over $\mathbb{Q}$, 
by the Chebotarev density theorem there are infinitely many prime numbers $p$ such that $m_j\in\mathbb{Z}_{(p)}[X]$ and $m_j$ has a zero modulo $p$.
We can therefore choose primes $p_1,\dots,p_{r}$ that are
\begin{enumerate}
\item pairwise distinct,
\item\label{it} greater than $i^{d_i^2}$,
\item not among the finitely many {\em bad primes} of the cover $\varphi_i$ (cf.~\cite[Def.~2.6]{Legrand}),
\item not among the finitely many prime numbers that ramify in $K_{i-1}$,
\item and such that $m_j$ has a zero modulo $p_j$ for $j=1,\dots,r$.
\end{enumerate} 
Now \cite[Corollary 3.3]{Legrand} gives $x\in\mathbb{P}^1(\mathbb{Q})$ such that 
the fiber $\varphi_i^{-1}(x)$ is irreducible with function
field a Galois extension $N_i$ of $\mathbb{Q}$ with ${\rm Gal}(N_i/\mathbb{Q})\cong{\rm Gal}(E/F)$
and such that the inertia group at each $p_j$ is generated by an element of $C_j$.
In particular, in each (not necessarily Galois) subextension $\mathbb{Q}\subsetneqq M\subseteq N_i$,
one of the $p_1,\dots,p_r$ ramifies.

Let $K=\bigcup_i K_i$.
By construction, $X_i(K)\supseteq X_i(N_i)\neq\emptyset$, so $K$ is pseudo-algebraically closed.
Moreover, $K$ satisfies (N) as the conditions of Corollary \ref{cor} are met:
Each $K_{i-1}\subsetneqq M\subseteq K_i=K_{i-1}N_i$ is of the form $M=K_{i-1}M_0$ for some
$\mathbb{Q}\subsetneqq M_0\subseteq N_i$,
and by construction there is a prime $p>i^{d_i^2}$ that ramifies in $M_0$ (and therefore in $M$) but not in $K_{i-1}$.

\begin{Remark}
Lukas Pottmeyer pointed out to me that
replacing $i^{d_i^2}$ in (\ref{it}) by $(i+1)^{4d_i^2}$ will achieve that 
$K_{\log(2)}=\{0,\pm1\}$, i.e.~$\alpha=2$ is of smallest positive
height in $K$.
\end{Remark}

\section{Proof of Proposition \ref{prop3}}

\noindent
For a prime number $p$ we denote by $\mathbb{Q}^{{\rm t}p}$ the field of {\em totally $p$-adic numbers},
i.e.~the maximal Galois extension of $\mathbb{Q}$ in which $p$ is totally split.
%For a set of primes $S$ we let $\mathbb{Q}^{{\rm t}S}=\bigcap_{p\in S}\mathbb{Q}^{{\rm t}p}$.
We first recall a result of Bombieri and Zannier \cite[Example 2]{BZ}.
They prove that for any finite set of prime numbers $p_1,\dots,p_n$, the intersection $L:=\bigcap_{i=1}^n\mathbb{Q}^{{\rm t}p_i}$ does not have (N).
More precisely, they show that
\begin{eqnarray}\label{eqn:BZ}
 \liminf_{\alpha\in L} h(\alpha) \leq \sum_{i=1}^n\frac{\log p_i}{p_i-1}.
\end{eqnarray}
To start our construction, fix any $T>0$ and choose a sequence $d_1,d_2,\dots$ 
such that $d_i>e$ for each $i$ and $\sum_{i=1}^\infty\frac{\log d_i}{d_i-1}<T$.
We want to construct an infinite sequence of primes $p_1,p_2,\dots$ and pairwise distinct elements $x_1,x_2,\dots\in\bigcap_{i=1}^\infty\mathbb{Q}^{{\rm t}p_i}$
with $p_i>d_i$ and $h(x_i)<T$ for each $i$.
Suppose we already constructed primes $p_1,\dots,p_{n-1}$ and $x_1,\dots,x_{n-1}\in\bigcap_{i=1}^{n-1}\mathbb{Q}^{{\rm t}p_i}$ 
with $p_i>d_i$ and $h(x_i)<T$ for $i=1,\dots,n-1$.
By the Chebotarev density theorem, there are infinitely many primes $p$ such that $p$ is totally split in 
the Galois closure of $\mathbb{Q}(x_1,\dots,x_{n-1})$,
in other words, $x_1,\dots,x_{n-1}\in\mathbb{Q}^{{\rm t}p}$.
Choose such a prime $p_n>d_n$ and note that $x_1,\dots,x_{n-1}\in\bigcap_{i=1}^{n}\mathbb{Q}^{{\rm t}p_i}$.
Now by (\ref{eqn:BZ}), 
there exists $x_n\in\bigcap_{i=1}^n\mathbb{Q}^{{\rm t}p_i}\setminus\{x_1,\dots,x_{n-1}\}$ with 
$$
 h(x_n)\leq\sum_{i=1}^n\frac{\log p_i}{p_i-1}\leq\sum_{i=1}^n\frac{\log d_i}{d_i-1}<T.
$$
Continuing this construction, we arrive at $K:=\bigcap_{i=1}^\infty\mathbb{Q}^{{\rm t}p_i}$
with $K_T\supseteq\{x_1,x_2,\dots\}$ infinite, so $K$ does not satisfy (N).
As $d_i\rightarrow\infty$, the set $\{p_1,p_2,\dots\}$ is infinite.

\section*{Acknowledgements}

\noindent
The author would like to thank Martin Widmer for drawing his attention to the circle of questions discussed here,
Lior Bary-Soroker for helpful discussion on Proposition \ref{prop2},
and Lukas Pottmeyer for very interesting exchange on the subject and helpful comments on a previous version.
These discussions and the writing of this note took place during the conference {\em Specialization problems in diophantine geometry} in July 2017 in Cetraro, and the author is grateful to Umberto Zannier and the organizers for the opportunity to attend this exciting and inspiring meeting.

\end{document}